\documentclass{amsart}
\usepackage{amssymb, amsthm, amsmath, amsfonts,amscd,bm,fancyhdr}
\usepackage{graphics}
\usepackage{hyperref}
\usepackage[all]{xy}
\usepackage{enumerate}
\usepackage[mathscr]{eucal}
\usepackage{bbm}
\usepackage{tikz}
\usetikzlibrary{decorations.pathreplacing,angles,quotes,knots}
\usepackage{parskip}
\usepackage{GWmacros}

\providecommand{\U}[1]{\protect\rule{.1in}{.1in}}
\def\theenumi{\arabic{enumi}}

\def\theenumii{\alph{enumii}}
\def\p@enumii{\theenumi.}

\def\theenumiii{\arabic{enumiii}}
\def\p@enumiii{(\theenumi)(\theenumii)}

\def\p@enumiv{\p@enumiii.\theenumiii}

\theoremstyle{plain}

\newtheorem{proposition}[theorem]{Proposition}

\numberwithin{equation}{section}

\theoremstyle{definition}

\newtheorem{definition}[theorem]{Definition}
\newtheorem{example}[theorem]{Example}

\newtheorem{remark}[theorem]{Remark}

\newtheorem{thmab}{Theorem}

\DeclareMathOperator{\FI}{FI}

\newcommand{\Sn}{\mathfrak{S}}

\renewcommand{\R}{\mathbb{R}}
\newcommand{\RR}{\mathbb{R}}

\newcommand{\X}{\mathcal{X}}

\newcommand{\dt}{\bullet}

\newcommand{\arXiv}[1]{\href{http://arxiv.org/abs/#1}{\nolinkurl{arXiv:#1}}}
\newcommand{\arXivV}[2]{\href{http://arxiv.org/abs/#1}{\nolinkurl{arXiv:#1v#2}}}

\title{Excessive symmetry can preclude cutoff}

\author[E.~Ramos]{Eric Ramos}
\address[E. Ramos]{Bowdoin College Department of Mathematics, Searles Hall, Brunswick, ME 04011}
\email{e.ramos@bowdoin.edu}

\author[G.~White]{Graham White}
\email{grahamwhite@alumni.stanford.edu}

\thanks{The first author was supported by NSF grants DMS-1704811 and DMS-2137628.}
\begin{document}

\begin{abstract}
For each $n,r \geq 0$, let $KG(n,r)$ denote the Kneser Graph; that whose vertices are labeled by $r$-element subsets of $n$, and whose edges indicate that the corresponding subsets are disjoint. Fixing $r$ and allowing $n$ to vary, one obtains a family of nested graphs, each equipped with a natural action by a symmetric group $\Sn_n$, such that these actions are compatible and transitive. Families of graphs of this form were introduced by the authors in \cite{RW}, while a systematic study of random walks on these families were considered in \cite{RW2}. In this paper we illustrate that these random walks never exhibit the so-called product condition, and therefore also never display total variation cutoff as defined by Aldous and Diaconis \cite{AD}. In particular, we provide a large family of algebro-combinatorially motivated examples of collections of Markov chains which satisfy some well-known algebraic heuristics for cutoff, while not actually having the property.
\end{abstract}


\maketitle

\section{Introduction}
In the paper \cite{RW2}, the authors considered random walks on a new kind of algebro-combinatorial objects: $\FI$-graphs. Formally speaking, an \textbf{$\FI$-graph} is a functor from the category of finite sets with injections to the category of (finite) graphs and graph homomorphisms. More concretely, one may think of an $\FI$-graph as a family of nested graphs $\{G_n\}$, each equipped with an action of the symmetric group $\Sn_n$, which is compatible with the inclusions $G_n \subseteq G_{n+1}$. Examples of these objects include the complete graphs $K_n$, the Kneser graphs $KG(n,r)$, and the Johnson graphs $J(n,r)$. They also include more exotic examples such as the graph of commuting transpositions on $\Sn_n$, as well as the graph of possible vertex colorings of a fixed graph. We will usually denote an $\FI$-graph by $G_\dt$. See Section \ref{examples} for some other examples.

In the prequel work \cite{RW2}, the authors briefly noted the fact that the family of simple random walks on $\FI$-graphs might not exhibit \textbf{cutoff}, in the sense of Aldous and Diaconis \cite{AD} (see Definition \ref{cutoffdef}). In brief, we say that a family of Markov chains $\{X_t^{(n)}\}_{n \geq 0}$ exhibits a cutoff so long as the time taken for them to move from being slightly mixed to very close to mixed is small compared to the time taken to achieve either of these things, for large enough $n$. Cutoff is seen to appear in many natural families of Markov chains, and has been a very active field of study since its inception in \cite{AD,D}.

In Diaconis's treatment \cite{D}, he notes that in many known cases where the cutoff phenomena is see to appear, there are certain algebraic restrictions on the spectrum of the chain. One such restriction, for instance, is that the second biggest eigenvalue has multiplicity that grows in $n$. Diaconis is then led to conjecture that this boundedness is a necessary condition for cutoff \cite{D}. Diaconis also notes in that work that the cutoff phenomenon seems considerably more likely in situations where the chain has an abundance of symmetry. This perspective was later reinforced by work of Lubetzky and Sly \cite{LS}, which displayed that families of Markov chains on random regular graphs exhibit cutoff.

From the perspective of $\FI$-graphs, if one were hoping to prove the appearance of the cutoff phenomenon, it would therefore seem most beneficial to limit oneself to situtations wherein symmetry is most apparent. In this work we will look at the class of \textbf{transitive} $\FI$-graphs. We say that an $\FI$-graph $G_\dt$ is transitive whenever the action of $\Sn_n$ on $G_n$ is vertex-transitive for all $n \gg 0$. All three of the examples of $\FI$-graphs given in the first paragraph are transitive. It is a fact (see Proposition \ref{eigengrow}) that the second biggest eigenvalue of a transitive $\FI$-graph has multiplicity which grows like a non-constant polynomial in $n$.  The main result of this paper is that, despite the aforementioned heuristics for cutoff, random walks on transitive $\FI$-graphs cannot display the phenomenon.

\begin{thmab}\label{mainthm}
Let $G_\dt$ be a transitive $\FI$-graph. Then the family of simple random walks on the graphs $G_n$ do not exhibit cutoff. (see Defintion \ref{cutoffdef}).
\end{thmab}

In his recent work \cite{L}, Lacoin constructed infinite families of Markov chains that do not have cutoff, despite satisfying the strong heuristic of the \textbf{product condition} (see Definition \ref{prodco}). In this paper, we will show that our families of Markov chains can never satisfy the product condition. Therefore, one can think of this work as being parallel to Lacoin's work, though our examples violate different heuristics.

In summary, the purpose of this paper is to display the following: There exist many algebro-combinatorially defined collections of graphs $\{G_n\}_{n \geq 0}$ such that the family of simple random walks on these graphs:
\begin{enumerate}
\item does not exhibit cutoff, or even the product condition;
\item is transitive, in that for each $n$ there is a vertex-transitive action of $\Sn_n$ on $G_n$ which preserves the probability measure;
\item satisfies the Diaconis eigenvalue heuristic \cite{D} for cutoff, in that the multiplicity of the second biggest eigenvalue of the transition matrix for the Markov chain is growing to $\infty$ with $n$.
\end{enumerate}

We will see in the proof of the main theorem that there is a very strong sense in which transitive $\FI$-graphs are \emph{too} symmetric to exhibit cutoff. This will be made precise in what follows.

\section*{Acknowledgments}
The authors would like to send their sincere thanks to David Levin and Bal\'azs Gerencs\'er for helpful discussions. The first author was supported NSF grants DMS-1704811 and DMS-2137628.

\section{Background}

In this section, we cover the majority of the background required to understand the results of this paper. Much of the exposition here is based on the prequel paper \cite{RW2}.

\subsection{Mixing times}
\label{sec:markovbackground}

We begin by briefly reviewing the theory of mixing times for Markov chains. Following this, we will spend some time recalling the notion of cutoff for families of Markov chains. All of what follows can be found in any standard text on the subject, such as \cite{LPW}. 

\begin{definition}\label{markovdef}
Let $\mathcal{X}$ be a finite set. Then a \textbf{Markov chain} on $\mathcal{X}$ is a family of random variables $\{X_t\}_{t = 0}^\infty$ such that for all $t \geq 0$, and all $(t+1)$--tuples $(x_0,\ldots,x_t) \in \mathcal{X}^{t+1}$,
\begin{enumerate}
\item $\mathbb{P}(X_t = x_t \mid X_{t-1} = x_{t-1}, \ldots , X_{0} = x_0) = \mathbb{P}(X_t = x_t \mid X_{t-1} = x_{t-1})$, and
\item $\mathbb{P}(X_t = x_t \mid X_{t-1} = x_{t-1}) = \mathbb{P}(X_{t-1} = x_t \mid X_{t-2} = x_{t-1}).$
\end{enumerate}

The information necessary to define a Markov chain is the state space $\X$ and the collection of transition probabilities --- the probabilities of moving from any state to any other. These probabilities are collected in the \textbf{transition matrix}, whose $(i,j)$--entry is the probability of moving from state $i$ to state $j$ in a single step. If $a,b \in \X$ are such that $P(a,b) > 0$, then we say that $b$ is a \textbf{neighbor} of $a$.

We say that a Markov chain $\{X_t\}_t$ on $\mathcal{X}$ is \textbf{connected} or \textbf{irreducible} if for any pair of states $x,y \in \mathcal{X}$ there is some $t > 0$ such that
\[
P^t(x,y) > 0
\]
The matrix $P$ is independent of the choice of initial distribution $\mathbb{P}(x) := \mathbb{P}(X_0 = x)$. We will usually interpret a choice of initial distribution as a row vector in $\RR \mathcal{X}$ whose coordinates sum to 1. A \textbf{stationary distribution of a Markov chain} is a choice of initial distribution $\pi$ having the property that $\pi \cdot P = \pi$.\\

Finally, we say that a Markov chain is \textbf{transitive} if there is a transitive action by some group $G$ on the state space $\mathcal{X}$, such that for all $g$ in $G$, and all $x,y \in \mathcal{X}$, $P(x,y) = P(gx,gy)$.
\end{definition}

\begin{theorem}[Proposition 1.14 and Corollary 1.17 of \cite{LPW}]
Let $(X_t,P)$ be a connected Markov chain on a state space $\mathcal{X}$. Then there exists a unique distribution $\pi$ such that $\pi \cdot P = \pi$.
\end{theorem}

\begin{remark}
It is easily verifiable that if $X_t$ is a transitive Markov chain with a unique stationary distribution, then that stationary distribution is uniform.
\end{remark}

Ultimately, the fundamental theorem of mixing times of Markov chains is that, with certain mild conditions, they eventually approach their stationary distribution. In order to talk about Markov chains approaching their stationary distributions, we will need to be able to measure the distance between distributions. For the work in this paper, we will follow the convention of using what is essentially the $L_1$ distance.

\begin{definition}
If $\mu$ and $\nu$ are two probability distributions on a set $\mathcal{X}$, then the \textbf{total variation distance} between $\mu$ and $\nu$ is the maximum value of $\mu(A) - \nu(A)$ over all events $A \subseteq \mathcal{X}$. Equivalently (for the finite chains we will consider), it is equal to the sum $$\sum_{x \in \mathcal{X}} \frac{1}{2}\left|\mu(x) - \nu(x)\right|.$$  
\end{definition}

\begin{Theorem}[Theorem 4.9 of \cite{LPW}]
Let $P$ be a Markov chain which is irreducible and aperiodic, with stationary distribution $\pi$. Then there exist constants $\al \in (0,1)$ and $C > 0$ so that for any starting state and any time $t$, the distance of the distribution after $t$ steps of $P$ from the stationary distribution $\pi$ is at most $C\al^t$. 
\end{Theorem}

This theorem requires that the Markov chain in question be aperiodic --- that it is not the case that all paths from a state to itself have length a multiple of any non-trivial period. 

\begin{definition} \label{mixdef}
Let $P$ be an irreducible and aperiodic Markov chain on the state space $\mathcal{X}$, and $\eps$ be any positive constant. The \textbf{mixing time} $\tme$ is the smallest time so that for any starting state $x \in \mathcal{X}$, the distribution after $\tme$ steps is within $\eps$ of the stationary distribution $\pi$. We also write $\tm := t_{\text{mix}}(1/4)$.
\end{definition}

\begin{Remark}
\label{rem:mixingeps}
Given a family of Markov chains indexed by $n$, we will sometimes want to say things like `These chains mix in a single step', or `These chains mix in five steps'. Statements like these should be understood to mean that for any $\eps$, there exists $N$ so that for all $n > N$, the claimed bound is true of $\tme$.  
\end{Remark}

In this paper, our Markov chains will take the form of random walks on (finite) graphs. For us, graphs will always be connected.

\begin{definition}\label{reverse}
If $(X_t,P)$ is a connected Markov chain on a state space $\mathcal{X}$ with stationary distribution $\pi$, then we say $(X_t,P)$ is \textbf{reversible} if for all $x,y \in \mathcal{X}$
\[
\pi(x)P(x,y) = \pi(y)P(y,x).
\]
\end{definition}

\emph{throughout this paper, we will assume that all Markov chains are irreducible, aperiodic, and reversible.} 

\subsection{Cutoff}

In this section, we outline the notion of cutoff first introduced by Aldous and Diaconis \cite{AD}. We also take the time to discus a variety of heuristics for when families are expected to exhibit a cutoff. This will be relevant later (Section \ref{examples}) when we construct examples which violate these heuristics.

\begin{definition}\label{cutoffdef}
Let $\{X^{(n)}_t\}_{n \geq 0}$ be a family of irreducible, aperiodic Markov chains. For each $n \geq 0$ and $\epsilon \in (0,1)$ we write $\tme^{(n)}$ for the mixing time of $X^{(n)}_t$. We say that the family \textbf{mixes in eventually constant time} if for all $\epsilon \in (0,1)$, $\tme^{(n)}$ is $O(1)$. We say that $\{X^{(n)}_t\}_{n \geq 0}$ \textbf{exhibits cutoff} if it does not mix in eventually constant time, and for all $\epsilon \in (0,1)$,
\begin{eqnarray}
\lim_{n \to \infty} \tme^{(n)}/t_{\text{mix}}(1-\epsilon)^{(n)} = 1. \label{cutoff}
\end{eqnarray}
\end{definition}

Intuitively, a family of Markov chains exhibits cutoff when the time between $\tm(1-\eps)$ and $\tm(\eps)$ is small compared to both of these quantities, for large enough $n$. When graphing total variation distance as a function of time, this describes a sudden drop from $1-\epsilon$ to $\epsilon$. Note that the usual definition of cutoff does not exclude chains with constant mixing time. For our purposes, chains with constant mixing time are not very interesting --- for instance, random walks on larger and larger complete graphs mix in a single step, so we exclude them and prove results about cutoff in chains not of this kind. See Remark \ref{notconst} for an instance where this is necessary.

Cutoff was introduced by Aldous and Diaconis in \cite{AD}. They were later expanded upon in an article of Diaconis \cite{D}. Since these original works, there has been an explosion of activity on the subject, propelled in part by the following contrast: cutoff is a natural condition that seems to hold for many classical examples of Markov chains (see \cite[Chapter 18]{LPW}, and the references therein), and yet it is also exceptionally hard to prove in almost every case of interest. While there are some general criteria for proving cutoff \cite{BHP}, the field has largely relied on more ad-hoc methods.

That being said, there are some heuristics which are generally believed to be good indicators of cutoff, though all are known to not necessarily imply cutoff. Interestingly, two of the most frequently used heuristics involve algebraic properties of the family. 

\begin{definition}
Recall that for a Markov chain $X_t$, one has an associated transition matrix $P$. Assuming that $X_t$ is irreducible, it is a fact that the largest eigenvalue of $P$ is 1. We write $\lambda$ for the second largest eigenvalue of $P$ in absolute value. The \textbf{relaxation time}, $t_{\text{rel}}$ of the Markov chain is the quotient
\[
t_{\text{rel}} := \frac{1}{1-\lambda}
\]
\end{definition}

\begin{proposition}[\cite{LPW}, Proposition 18.4]\label{prodco}
Let $\{X^{(n)}_t\}_{n \geq 0}$ be a family of aperiodic, connected, reversible Markov chains. Writing $\tm^{(n)}$ and $t_{\text{rel}}^{(n)}$ for the mixing and relaxation times of $X^{(n)}_t$, respectively, then
\begin{eqnarray}
t_{\text{rel}}^{(n)} = o(\tm^{(n)})\label{preco}
\end{eqnarray}
whenever $\{X^{(n)}_t\}_{n \geq 0}$ exhibits cutoff.
\end{proposition}

\begin{remark}\label{notconst}
Note that this proposition is dependent on our assumption that the family eventually mixes in non-constant time. Indeed, consider the simple random walk on the complete graph $K_n$. In this case, $\tme^{(n)} = 1$ for all $n \gg 0$ and all $\epsilon \in (0,1)$. In particular, this family satisfies the required limit (\ref{cutoff}). On the other hand, one easily computes that $t_{\text{rel}}^{(n)} = \frac{n-1}{n-2} \neq o(1)$.
\end{remark}

The condition (\ref{preco}) is significant enough, that we give it a name.

\begin{definition}
We say that a family of Markov chains $\{X^{(n)}_t\}_{n \geq 0}$ satisfies the \textbf{product condition}, if (\ref{preco}) holds.
\end{definition}

The product condition is generally seen as a strong indicator that the family in question exhibits cutoff. For instance, it is known that these conditions are equivalent for random walks on weighted trees \cite{BHP}. Also, in their seminal work, Basu, Hermon, and Peres examine a hitting-time condition that, when paired with the product condition, is equivalent to cutoff \cite{BHP}. That being said, however, the product condition is not equivalent to cutoff (see the examples in \cite[Chapter 18]{LPW}, due to Aldous and Pak). The main result of this work will show that random walks on so-called transitive $\FI$-graphs (see Section \ref{FIsection}) never satisfy the product condition. One of the main tools we use to prove this is the following well known pair of bounds (see \cite[Theorems 12.4 and 12.5]{LPW})

\begin{theorem}\label{maininequal}
Let $P$ be the transition matrix of a reversible, irreducible Markov chain with state space $\X$ and stationary distribution $\pi$. Writing $\pi_{\min} = \min_{x \in \X} \pi(x)$, we have
\[
(t_{\text{rel}}-1)\log(\frac{1}{2\epsilon}) \leq \tme \leq t_{\text{rel}}\log(\frac{1}{\epsilon\pi_{\min}})
\]
\end{theorem}

A second heuristic is due to Diaconis \cite{D}, and considers the multiplicity of the eigenvalue $\lambda$. Diaconis notes that if $\{X^{(n)}_t\}_{n \geq 0}$ is a family of irreducible, aperiodic Markov chains, then cutoff seems to be caused by an abundance in the multiplicity of the second largest eigenvalue $\lambda(n)$. Namely, whenever the function
\[
n \mapsto \text{multiplicity of }\lambda(n)
\]
goes to infinity with $n$, one should expect that the corresponding family of Markov chains exhibits cutoff.

In the present work, we will consider random walks in certain families of highly symmetric graphs. Our main result will show that these walks never exhibit the product condition. On the other hand, it will be shown that these walks always do satisfy the multiplicity heuristic of Diaconis, making it particularly interesting that cutoff is not present.

\subsection{$\FI$-sets and relations}\label{FIsection}

In this section we review the theory of FI-sets and relations first explored by the authors and Speyer in \cite{RSW}. This theory was heavily inspired by, and ultimately rests on the shoulders of, the theory of representation stability \cite{CEF}.

\begin{definition}\label{fiset}
We write $\FI$ to denote the category whose objects are the sets $[n] = \{1,\ldots,n\}$, and whose morphisms are injective maps of sets. An \textbf{$\FI$--set} is a functor $Z_\dt$ from $\FI$ to the category of finite sets. If $Z_\dt$ is a $\FI$--set, and $n$ is a non-negative integer, we write $Z_n$ for its evaluation at $[n]$. If $f:[n] \hookrightarrow [m]$ is an injection of sets, then we write $Z(f)$ for the map induced by $Z_\dt$.

An \textbf{$\FI$-subset}, or just a \textbf{subset}, of an $\FI$-set $Z_\dt$ is an $\FI$-set $Y_\dt$ for which there exists a natural transformation $Y_\dt \rightarrow Z_\dt$ such that $Y_n \hookrightarrow Z_n$ is an injection for all $n \geq 0$.
\end{definition}

While the above definition might appear somewhat abstract, one thing we hope to impress upon the reader is that one can think about these objects in quite concrete terms. To see this, first observe that for each $n$, $Z_n$ carries the natural structure of an $\Sn_n$-set, induced from the endomorphisms of $\FI$. With this in mind, one may therefore think of an $\FI$-set $Z_\dt$ as a sequence of $\Sn_n$-sets $Z_n$, which are compatible with one another according to the actions of the morphisms of $\FI$.

As one might expect, it is in the best interest of the theory to restrict our attention to a particular class of ``well-behaved'' $\FI$-sets. To this end we have the following definition.\\

\begin{definition}\label{fingen}
An $\FI$--set $Z_\dt$ is said to be \textbf{finitely generated in degree $\leq d$}, if for all $n \geq d$, one has
\[
Z_{n+1} = \bigcup_{f} Z(f)(Z_n)
\]
where the union is over all injections $f:[n] \hookrightarrow [n+1]$.
\end{definition}

The proof of the following theorem can be found in \cite{RSW}.

\begin{theorem}[\cite{RSW}, Theorem A]\label{mainstructurethm}
Let $Z_\dt$ denote an $\FI$-set which is finitely generated in degree $\leq d$. Then there exists a finite collection of integers $m_i \leq d$, and subgroups $H_i \subseteq \Sn_{m_i}$, such that, for $n$ sufficiently large, we have an isomorphism 
\[
Z_n \cong \bigsqcup_i \Sn_n / (H_i \times \Sn_{n-m_i})
\]
as sets with an action of $\Sn_n$.
\end{theorem}

In the paper \cite{RSW}, where $\FI$-sets were first examined, it is argued that many naturally occurring examples of $\FI$-sets come equipped with a collection of $\Sn_n$-equivariant relations. To be more precise, one has the following definition.

\begin{definition}\label{relation}
Let $Z_\dt$ and $Y_\dt$ denote two $\FI$-sets. Then the product $Z_\dt \times Y_\dt$ carries the structure of an $\FI$-set in a natural way. A \textbf{relation} between $Z_\dt$ and $Y_\dt$ is a subset $R_\dt$ of $Z_\dt \times Y_\dt$. If $Z_\dt = Y_\dt$, then we say that $R_\dt$ is a \textbf{relation on $Z_\dt$}

Given a relation $R_\dt$ between $Z_\dt$ and $Y_\dt$ we obtain a family of $\Sn_n$-linear maps
\[
r_n: \RR Z_n \rightarrow \RR Y_n
\]
where $\RR Z_n$ is the $\RR$-linearization of the set $Z_n$, and similarly for $\RR Y_n$. Properties of these maps were a major focus of \cite{RSW}. In this work, they will naturally arise as probability transition matrices of certain families of Markov chains.
\end{definition}

It is a fact, proven in \cite{RSW}, that any relation between two finitely generated $\FI$-sets is itself finitely generated. It can be proven from this that, if $Z_\dt$ is a finitely generated $\FI$-set, then the number of $\Sn_n$-orbits of pairs $(x,y) \in Z_n \times Z_n$ is eventually independent of $n$ (see \cite{RSW}). Perhaps the most notable classes of examples of $\FI$-set relations arise in the theory of $\FI$-graphs.

In \cite{RW}, the authors defined what they called $\FI$-graphs, functors from $\FI$ to the category of graphs and graph homomorphisms. In this case, one may think of an $\FI$-graph as an $\FI$-set of vertices paired with a symmetric relation dictating how these vertices are connected through edges. One should note in this case that the associated linear maps $r_n$ are what one would usually call the \textbf{adjacency matrices} of the corresponding graphs.

\begin{remark}
Recall that, in this paper, a graph is connected by definition.
\end{remark}

Some examples of $\FI$-graphs include the complete graphs $K_n$, whose vertices are the set $[n]$, and whose associated relation is comprised of all pairs $(i,j)$ with $i \neq j$, and the Kneser graphs $KG(n,r)$, whose vertices are given by $r$-element subsets of $n$ and whose associated relation is comprised of all pairs $(A,B)$ such that $A \cap B = \emptyset$. We will see other examples of $\FI$-graphs throughout the work (Section \ref{examples}).

For the remainder of this paper, the primary objects of study will be finitely generated $\FI$-graphs and, more specifically, simple random walks on these objects. In the prequel paper \cite{RW2}, a theory is developed for what the authors call models of random walks on $\FI$-graphs. This more general theory includes things such as lazy modifications of the simple walk. Going forward we limit our exposition to simple random walks, though everything we prove will work in the more general setting.

The following theorem follows from \cite[Corollary C]{RSW}.

\begin{theorem}[\cite{RSW}, Corollary C]\label{eigenvaluestab}
Let $G_\dt$ denote a finitely generated $\FI$-graph, with vertex $\FI$-set $V_\dt$ and edge relation $E_\dt$. Write $P_n$ for the transition matrix of the simple random walk on $G_n$. Then,
\begin{enumerate}
\item the number of distinct eigenvalues of $P_n$ is independent of $n$ for $n \gg 0$;
\item there exists a finite list $\{f_i\}$ of functions which are algebraic over $\Q(n)$ and real valued, for which $\{f_i(n)\}$ is the complete list of eigenvalues of $P_n$ for $n \gg 0$;
\item for any $f_i$ as in the previous part, the function
\[
n \mapsto \text{ the algebraic multiplicity of $f_i(n)$ as an eigenvalue of $P_n$}
\]
agrees with a polynomial for $n \gg 0$.
\end{enumerate}
\end{theorem}

In the next section we will relate the conclusions of this theorem with the Diaconis cutoff heuristic.

\section{Rational transitions between Markov chains}

In this section we introduce the concept of a rational transition between Markov chains. Inutitively, these are circumstances where one imagines going from one Markov chain on a state space $\X$ to another by deforming the transition matrix by rational functions.

\begin{definition}
Let $\X$ be a finite set, and let $(X_t,P)$ and $(Y_t,Q)$ be two Markov chains on $\X$. Then a \textbf{rational transition from  $(X_t,P)$ to $(Y_t,Q)$} is a family of Markov chains $\{(X^{(n)}_t,P^{(n)})\}_{n \geq 0}$ such that:
\begin{enumerate}
\item $P^{(n)}$ is a matrix with coefficients in the field of rational functions $\R(n)$;
\item $P^{(0)} = P$, and $\lim_{n \to \infty} P^{(n)}$ exists and agrees with $Q$.
\end{enumerate}
\end{definition}

\begin{example}
Let $p,q \in (0,1)$, and let $\X = \{x,y\}$. Then we have two Markov chains on $\X$ given by $P = \begin{pmatrix}p & 1-p\\ 1-p & p\end{pmatrix}$, and $Q = \begin{pmatrix}q & 1-q\\ 1-q & q\end{pmatrix}$. Then one possible rational transition between $P$ and $Q$ is the family of Markov chains with transition matrices given by
\[
P^{(n)} = \begin{pmatrix} \frac{1}{n+1}p + \frac{n}{n+1}q & \frac{1}{n+1} (1-p)+ \frac{n}{n+1}(1-q) \\ \frac{1}{n+1}(1-p) + \frac{n}{n+1}(1-q) & \frac{1}{n+1}p + \frac{n}{n+1}q \end{pmatrix}
\]
\end{example}

In this paper, rational transitions between Markov chains will naturally arise in the context of orbit walks associated to $\FI$-graphs.\\

For our purposes, it will be important to ask the question how do various statistics such as relaxation time and mixing time vary during a rational transition? The first of these quantities can be answered through simple linear algebra.

\begin{proposition}\label{relaxalgebra}
Let $\{(X^{(n)}_t,P^{(n)})\}_{n \geq 0}$ be a rational transition between two Markov chains on a state space $\X$. Then, for $n \gg 0$, the function,
\[
n \mapsto t_{rel}^{(n)}
\]
agrees with a function which is algebraic over the field $\R(n)$.
\end{proposition}

\begin{proof}
The matrix $P^{(n)}$ is a $|\X| \times |\X|$ square matrix over the field $\R(n)$. It follows that the eigenvalues of $P^{(n)}$ are algebraic over $\R(n)$, as they satisfy the characteristic polynomial of $P^{(n)}$. Note that by assumption our Markov chains are reversible, so we may also assume that these eigenvalues are real valued. We denote these eigenvalues by $\lambda(n)$ in what follows.

It remains to argue that $\max_{\lambda \neq 1}\{|\lambda(n)|\}$ is algebraic for $n \gg 0$. Indeed, if $\lambda(n)$ is algebraic with real values, then it only assumes the value 0 finitely many times. In particular, for $n \gg 0$, $\lambda(n)$ is of fixed sign. Therefore, $|\lambda(n)|$ is in agreement with an algebraic function (i.e. either $\lambda(n)$ or $-\lambda(n)$) for $n \gg 0$. A similar argument then implies that the maximum $\max_{\lambda \neq 1}\{|\lambda(n)|\}$ is uniquely achieved by a single $|\lambda(n)|$ for $n \gg 0$, as the difference of two algebraic functions is still algebraic. This completes the proof.
\end{proof}

Resolving the eventual behavior of the mixing time $\tme^{(n)}$ is a bit more subtle to contend with. Considering that our concern is mostly in its behavior in the large $n$ limit, we will rely on a hitting time approximation due to Peres and Sousi \cite{PeSo}.

\begin{definition}
Let $(X_t,P)$ denote a Markov chain on a state space $\X$, with stationary distribution $\pi$. Then for $\alpha \in (0,1/2)$, the $\textbf{$\alpha$-large-set hitting time}$ is the quantity,
\[
t_{hit}(\alpha) = \max_{A \subseteq \X, x\in X, \pi(A) \geq \alpha} E_x[\tau_A]
\]
where $E_x[\tau_A]$ is the expected time for the Markov chain to enter the set $A$, conditional on it beginning at $X_0 = x$.
\end{definition}

intuitively, one should not expect the Markov chain to have mixed before it is able to hit big sets. Remarkably, however, there is a sense in which the converse is true as well. This is summarized in the following theorem.

\begin{theorem} [Peres and Sousi, \cite{PeSo}] \label{hitismix}
Let $(X_t,P)$ be a Markov chain on a state space $\X$. Then for all $\alpha \in (0,1/4)$, there exist constants $c_\alpha,c'_\alpha$ such that,
\[
c_\alpha t_{hit}(\alpha) \leq \tm \leq c'_\alpha t_{hit}(\alpha).
\]
Importantly, $c_\alpha$ and $c'_\alpha$ depend only on $\alpha$, and not $P$ or $\X$.
\end{theorem}

\begin{remark}
Once again, reversibility of $(X_t,P)$ is critical for the above theorem.
\end{remark}

Theorem \ref{hitismix} can be thought of as saying that $\alpha$-large-set hitting times are essentially the same as mixing times. The extra information that the constants $c_\alpha$ and $c'_\alpha$ do not depend on the process itself will allow us to use these inequalities in entire families of Markov chains. In particular, it will allow us to resolve the question of the kinds of growth that mixing times of rational transitions can attain.\\

\begin{theorem}\label{RatMixing}
Let $\{(X_t^{(n)},P^{(n)}\}_{n \geq 0}$ be a rational transition between Markov chains on a state space $\X$. Then the function
\[
n \mapsto \tm^{(n)}
\]
is $\Theta(f(n))$, where $f(n) \in \R(n)$. That is to say, there exists $f(n) \in R(n)$ as well as constants $\beta,\gamma$ such that for all $n \gg 0$
\[
\gamma f(n) \leq \tm^{(n)} \leq \beta f(n).
\]
\end{theorem}

\begin{proof}
By Theorem \ref{hitismix}, it will suffice to find some $\alpha \in (0,1/2)$ such that $t_{hit}^{(n)}(\alpha)$ is a rational function for $n \gg 0$. Indeed, we will show this is the case for all $\alpha \in (0,1/2)$.

We will first show that, for any $x \in \X$ and $A \subseteq \X$, the function
\[
n \mapsto E_x^{(n)}[\tau_A]
\]
is in agreement with a rational function for all $n \geq 0$. To see this, let $V$ be the $\R(n)$-vector-space with basis in bijection with those elements of $\X$ not in $A$. Then by the usual first-step recursion satisfied by hitting times, we find $E_x[\tau_A]$ to be the $x$-coordinate of the vector $Q \in V$, defined by the matrix equation
\[
(I-P^{(n)}|_V)Q = \mathbf{1}
\]
where $\mathbf{1}$ is the vector $\sum_{y \notin A}y$. The matrix $I-P^{(n)}$ is generally not invertible, however the minor consisting of only those rows and columns corresponding to elements not in $A$ is invertible. In particular, $Q$ is unique and well-defined. This concludes the proof of our first claim.

Moving on, we first recall that,
\[
t_{hit}^{(n)}(\alpha) = \max_{A \subseteq \X, x\in X, \pi^{(n)}(A) \geq \alpha} E_x[\tau_A].
\]
Observe that the set $S^{(n)}_\alpha := \{A \subseteq \X \mid \pi^{(n)}(A) \geq \alpha\}$ is independent of $n$ whenever $n \gg 0$, as $\pi^{(n)}$ is an element of an $\R(n)$ vector space. Just as we argued in the proof of Proposition \ref{relaxalgebra}, it follows from the previous paragraphs and our observation about $S_\alpha^{(n)}$ that the maximum of the set
\[
\{E^{(n)}_x[\tau_A] \mid x \in \X, A \in S^{(n)}_\alpha\}
\]
is achieved by a choice of $A$ and $x$ which is unchanging in $n$, whenever $n \gg 0$. This concludes the proof.
\end{proof}

\section{Walks on transitive $\FI$-graphs}

In this section, we discuss useful properties held by what we call transitive $\FI$-graphs. This leads into the next section wherein we conclude by proving our main theorem.

\begin{definition}
Let $G_\dt$ be a finitely generated $\FI$-graph with vertex $\FI$-set $V_\dt$ and edge relation $E_\dt$. We say that $G_\dt$ is \textbf{transitive} if, for all $n \gg 0$, the action of $\Sn_n$ on $V_n$ is transitive.
\end{definition}

Examples of transitive $\FI$-graphs include the complete graphs $K_n$, and the Kneser graphs $K(n,r)$. We will see many more examples later (Section \ref{examples}).

\begin{proposition}\label{notbipartite}
Let $G_\dt$ be a transitive $\FI$-graph. Then for $n \gg 0$, $G_n$ is not bipartite. In particular, the simple random walk on $G_n$ is aperiodic for $n \gg 0$.
\end{proposition}

\begin{proof}
We prove this proposition using the spectral characterization of connected bipartite graphs. Let $G$ be a graph, and assume adjacency matrix of $G$ has distinct eigenvalues 
\[
\lambda_1 > \lambda_2 > \ldots > \lambda_m
\]
Then $G$ is bipartite if and only if, for each $i$, $\lambda_i = -\lambda_{m-i+1}$, and the multiplicity of $\lambda_i$ is equal to that of $\lambda_{m-i+1}$.

Because our graphs are connected, this implies that if $G_n$ is bipartite then both $\lambda_1$ and $\lambda_m$ appear with multiplicity 1. In our setting, the action of the symmetric group commutes with the adjacency matrix and therefore the eigenspaces also carry an action of the symmetric group. This shows that the eigenspaces of $\lambda_1$ and $\lambda_m$ must either be isomorphic to the trivial representation or the sign representation. Representation stability theory implies that the sign representation cannot appear for $n \gg 0$ \cite{CEF}, while transitivity implies that there cannot be more than one copy of the trivial representation. This shows that $G_n$ cannot be bipartite.
\end{proof}

\begin{remark}
To adapt the above proof to cases of more general models of random walks on $\FI$-graphs, one simply notes that if a random walk is periodic with period $d$, then the $d$-step walk is aperodic and reducible. This would imply that spectrum of the $d$-step walk has multiple eigenvalues of multiplicity 1, whence the above argument leads to a contradiction.
\end{remark}

The above proposition is useful, as it allows us to essentially ignore possible issues with periodicity. Our second result is more related to the heuristics of Diaconis. In particular, we will find that the second largest eigenvalue of a transitive $\FI$-graph must grow non-trivially with $n$. Note the similarity in the style of proofs between the following proposition and the previous one.

\begin{proposition}\label{eigengrow}
Let $G_\dt$ be a transitive $\FI$-graph, and let $P_n$ denote the transition matrix of the simple random walk on $G_n$. If $\lambda(n)$ is the second largest eigenvalue of $P_n$, then the multiplicity of $\lambda(n)$ agrees with a non-constant polynomial in $n$ for $n \gg 0$.
\end{proposition}

\begin{proof}
The main theorem of \cite{RSW} implies that the multiplicity of $\lambda(n)$ eventually agrees with a polynomial in $n$. It therefore remains to argue that this polynomial is non-constant. Because the eigenspaces of $P_n$ carry an action by the symmetric group, we know that the irreducible constituents that appear must obey the restrictions imposed by representation stability. In particular, the only way that the dimension of this eigenspace is constant is if it decomposed into a sum of trivial representations. Because our symmetric group action is transitive, and because the consequently unique copy of the trivial representation is being occupied by the eigenspace for the eigenvalue 1, it follows that the dimension must be growing.
\end{proof}

We next turn our attention to the orbit graphs associated to a transitive $\FI$-graph. 

\begin{definition}
Let $G_\dt$ denote a transitive $\FI$-graph, and for some $m \gg 0$, fix a vertex $x \in G_m$. for each $n \geq m$, we write $x(n) \in V(G_n)$ to denote the image of $x$ under the map induced by the standard inclusion $\iota: [m] \hookrightarrow [n]$. Then the \textbf{$x$-roofed orbit graph $G_n^x$} associated to $G_n$ is the graph whose vertices are indexed by $\Sn_n$-orbits of pairs of vertices of the form $[y, x(n)]$. Two orbits $[y,x(n)],$ and $[z,x(n)]$ are connected by an edge if and only if there exists $(z',x(n)) \in [z,x(n)]$ such that $z'$ is adjacent to $y$. Note that, from the discussions of Section \ref{FIsection}, the graph $G_n^x$ is eventually independent of $n$. We define a random walk on $G_n^x$ by the transition rule
\[
P^x_n([y,x(n)],[z,x(n)]) = \sum_{[z',x(n)] =[z,x(n)]} P_n(y,z'),
\]
where $P_n(y',z)$ is the transition rule for the simple random walk on $G_n$. Hence-forth we will refer to this random process as the \textbf{orbit walk of $G_\dt$}.
\end{definition}

An alternative description of the orbit walk of $G_\dt$ is given in terms of the decomposition of the vertex set given by Theorem \ref{mainstructurethm}. In particular, the elements of $V(G_n)$ may be written as equivalence classes of permutations $[\sigma] \in \Sn_n / H \times \Sn_{n-m}$, where $H \leq \Sn_m$ and $m \geq 0$. It is easily seen that the classes $[\sigma]$ are in bijection with ordered tuples $(S_{\alpha_1},\ldots,S_{\alpha_r})$, where the $\alpha_i$ are the $H$-orbits of $[m]$, and the $S_{\alpha_i} \subseteq [n]$ are disjoint with $\sum_i |S_i| = m$. Indeed, for an $H$-orbit $\alpha_i$, one has $S_{\alpha_i} = \sigma(\alpha_i)$. We call the $S_{\alpha_i}$ the \textbf{labels} of the associated vertex. With regards to this description, an orbit of pairs of two vertices $[x,y]$ can then be described by indicating how much overlap exists in the labels of $x$ and $y$, respectively.

In particular, having fixed our vertex $x$ as in the definition of the orbit graph, the vertices of $G_n^x$ can be thought of as indicating how different the labels of the corresponding vertex are from those of $x(n)$.

We now take the time to record the following important proposition.

\begin{proposition}\label{isRatTran}
Let $G_\dt$ denote a transitive $\FI$-graph, and let $m \gg 0$ be so large that $G_n^x$ is unchanging for all $n \geq m$. Then the family of Markov chains $\{(X_t^{(n)},P^x_n)\}_{n \geq m}$ is a rational transition between $P^x_m$ and a Markov chain whose stationary distribution is given by 
\[
\pi_{\infty}([y,x(n)]) = \begin{cases} 0 &\text{ if $y$ and $x(n)$ have any overlap in their labels}\\ 1 &\text{ if $y$ and $x(n)$ have totally disjoint labels.}\end{cases}
\]
\end{proposition}

\begin{proof}
The first claim will follow once we know that $P^x_n \in \R(n)$. This was proven in \cite{RW2}. For the second claim, we note that the proportion of vertices which share no labels with $x(n)$ is approaching 1 as $n \to \infty$.
\end{proof}

\section{The proof of the main theorem}

We begin by making explicit the relationship between the mixing times of a model of a random walk on a transitive $\FI$-graph, and the mixing times of the associated orbit graph.

\begin{theorem}\label{TransitiveMixing}
Let $G_\dt$ denote a transitive $\FI$-graph, and let $m \gg 0$ be so large that $G_n^x$ is unchanging for all $n \geq m$. If we write $\tme^{(n)}$ for the mixing time for the simple random walk on $G_n$, and $\tme^{(n)}_x$ for the mixing time for the orbit walk, then $\tme^{(n)}_x = \tme^{(n)}$.
\end{theorem}

\begin{proof}
Write $\pi^{(n)}$ and $\pi^{(n)}_x$ for the stationary distributions of the simple random walk on $G_n$ and the orbit walk, respectively.

We first note, for any $y \in V_n$,
\[
\pi^{(n)}_x([y,x(n)]) = \sum_{[y',x(n)]=[y,x(n)]} \pi^{(n)}(y') = |[y,x(n)]|\cdot\pi^{(n)}(y),
\]
as the stationary distribution is uniform for transitive chains. On the other hand, for any $y \in V_n$,
\[
(P^x_n)^t([x(n),x(n)],[y,x(n)]) = \sum_{[y',x(n)] = [y,x(n)]} P_n^t(x(n),y') = |[y,x(n)]|\cdot P_n^t(x(n),y),
\]
as the condition that $[y,x(n)] = [y',x(n)]$ implies that there is a permutation $\sigma$ that sends $y$ to $y'$ while fixing $x(n)$ and therefore $ P_n^t(x(n),y') = P_n^t(\sigma x(n),\sigma y') = P_n^t(x(n),y)$. Putting these two together we obtain,

\begin{align*}
\sum_{[y,x(n)]} |\pi_x^{(n)}([y,x(n)]) - (P^x_n)^t([x(n),x(n)],[y,x(n)])| &= \sum_{[y,x(n)]} \left( |[y,x(n)]| \cdot|\pi^{(n)}(y) - P_n^t(x(n),y)|\right)\\
																		&= \sum_y |\pi^{(n)}(y) - P_n^t(x(n),y)|
\end{align*}

Because the simple random walk on $G_\dt$ is transitive, its total variation distance to stationary can be calculated with respect to any starting position. It follows that the total variation distance between $P_n^t$ and $\pi^{(n)}$ is no larger than the distance between $(P_n^x)^t$ and $\pi^{(n)}_x$. In particular, $\tme^{(n)}_x \geq \tme^{(n)}$. On the other hand, the orbit chain is clearly a projection of the simple random walk, whence $\tme^{(n)}_x \leq \tme^{(n)}$ from well known facts about projection chains.
\end{proof}

We are now ready to prove Theorem \ref{mainthm}. We will then conclude the paper by providing a collection of examples of transitive FI-graphs.

\begin{proof}[Proof of Theorem \ref{mainthm}]
For the remainder of this proof, fix a transitive $\FI$-graph $G_\dt$, as well as the associated orbit graph $G_\dt^x$. We will show that the family of simple random walks on $G_\dt$ do not satisfy the product condition.

Theorem \ref{TransitiveMixing} tells us that the mixing time of the simple random walk on $G_n$ $\tm^{(n)}$ equals that of the associated orbit walk. It now follows from Proposition \ref{isRatTran} and Theorem \ref{RatMixing} that $\tm^{(n)}$ is $\Theta(f(n))$, where $f(n)$ is some rational function. In particular, the quotient $\frac{\tm^{(n)}}{t_{rel}^{(n)}}$ is $\Theta$ of a algebraic function according to Theorem \ref{eigenvaluestab}. On the other hand, adapting Theorem \ref{maininequal} to our setting we see that
\[
\frac{\tm^{(n)}}{t_{rel}^{(n)}} \leq \log(4|G_n|),
\]
where $|G_n|$ is the number of vertices in $G_n$. In particular, as the number of vertices of a finitely generated $\FI$-graph grows as a polynomial in $n$, we see that $\frac{\tm^{(n)}}{t_{rel}^{(n)}}$ is bounded from above by something that is $\Theta$ of $\log(n)$. We finish this proof by showing that this bound implies that $\frac{\tm^{(n)}}{t_{rel}^{(n)}}$ is not limiting to infinity, whence the product condition fails.

Write $A(n)$ for the algebraic function that describes the end behavior of $\frac{\tm^{(n)}}{t_{rel}^{(n)}}$, and assume for contradiction that $\lim_{n \to \infty} A(n) = \infty$. By definition, we can find polynomials $p_0(n),\ldots,p_r(n)$ such that
\[
p_r(n)A^r(n) + \ldots + p_1(n)A(n) + p_0(n) = 0.
\]
Let $p_j$ denote a polynomial with highest degree among the $p_i$, and assume that $j$ is the largest such index with this property. Then we have,
\[
p_r(n)A^r(n) + \ldots + p_1(n)A(n) + p_0(n) = p_j(n)A^j(n)\left(1 + \sum_{i \neq j} \frac{p_i}{p_j}A^{i-j}\right) = 0
\]
Taking the limit as $n \to \infty$,
\[
\lim_{n \to \infty}  p_j(n)A^j(n)\left(1 + \sum_{i \neq j} \frac{p_i(n)}{p_j(n)}A^{i-j}(n)\right) = \lim_{n \to \infty}  p_j(n)A^j(n)\left(1 + \sum_{k \neq j} \frac{p_k(n)}{p_j(n)}A^{k-j}(n)\right),
\]
where the latter sum is over all indices $k$ such that $\deg(p_k) = \deg(p_j)$. This follows from the fact that $A$ is $O(\log(n))$, whence it and all of its powers limit to zero when divided by $n$. By assumption we have that $k < j$ and therefore,
\[
0 = \lim_{n \to \infty}  p_j(n)A^j(n)\left(1 + \sum_{k \neq j} \frac{p_k(n)}{p_j(n)}A^{k-j}(n)\right) = \lim_{n \to \infty}  p_j(n)A^j(n) = \infty,
\]
where we have (twice) used the fact that $\lim_{n \to \infty}A(n) = \infty$. This is our desired contradiction.
\end{proof}

\section{Examples of random walks}\label{examples}

In this section, we give some examples of random walks on transitive $\FI$-graphs. In all cases, tuples contain elements of $\{1,2,\dots,n\}$ without repetition. 

\begin{example}
States are unordered pairs and a step is to replace both numbers with any others.
\end{example}

\begin{example}
States are triples and a step is to replace a random entry with any unused number.
\end{example}

\begin{example}
States are triples and a step is to either randomly permute the labels, or replace the last one with any unused number.
\end{example}

\begin{example}
States are $k$--tuples and a step is to shift all entries either one place to the left or one place to the right, deleting the entry on that end and introducing a new one at random at the other.
\end{example}

Walks such as these have constant mixing time, because what needs to happen for them to be mixed can be described without reference to $n$ --- the first mixes in a single step, the second after all three coordinates have been chosen at least once each, the third once each initial label has been moved into the last position and then replaced, and the last once either `left' or `right' has been chosen $k$ more times than the other. This non-dependence on $n$ should feel like a very $\FI$-flavored property, and it relies on the fact that any graph of this kind has a description in terms of tuples of labels where adjacency depends only on which labels are equal or unequal to which other labels, never on what those labels actually are (see the classification theorem of \cite{RSW} and further discussion in \cite{RW2}). That is, whenever a random walk step introduces a new label, it is only as ``a random new label'', never specifying which label to use.

It is also possible to produce walks with mixing time longer than constant, by introducing multiple orbits of edges.

\begin{example}
States are pairs and a step is to replace the first entry with probability $\frac{1}{n}$ or the second with probability $\frac{n-1}{n}$.
\end{example}

\begin{example}
\label{ex:mult}
States are triples and a step is to replace the first entry with probability $\frac{1}{n}$ or the second and third with probability $\frac{n-1}{n}$.
\end{example}

\begin{example}
States are triples and a step is to either randomly permute the labels with probability $\frac{n-1}{n}$, or replace the last one with any unused number, with probability $\frac{1}{n}$.
\end{example}

In these three examples, the necessary conditions for mixing are still phrased without any dependence on $n$ --- in all cases, they are that we must wait until all labels have been replaced --- but the time for this is now linear in $n$, and similar constructions could produce mixing times of any higher power of $n$. While this dependence on $n$ may appear artificial, it can at least arise naturally from the construction of a simple random walk --- for instance, Example \ref{ex:mult} is essentially a simple random walk on the graph where the triple $\{a,b,c\}$ is connected to each $\{x,b,c\}$ and each $\{a,y,z\}$ --- there are just about $n$ times as many edges of the second kind as of the first.

Indeed, by the classification theorem of \cite{RSW}, any simple random walk on a transitive $\FI$-graph is of the forms described here --- the state space is $k$--tuples, perhaps with some identification, and moves involve reordering entries of the tuple and/or replacing some with random other elements. Because we are working with simple random walks, any reordering move implies that the reverse move is also possible and equally likely.


\begin{thebibliography}{aaaa}
\small
\bibitem[AD]{AD} D. Aldous and P. Diaconis, \textit{Shuffling cards and stopping times}, The American Mathematical Monthly, 93(5), 333-348.
\bibitem[BHP]{BHP} R. Basu, J. Hermon, and Y. Peres, \textit{Characterization of cutoff for reversible Markov chains}, Ann. Probab. Volume 45, Number 3 (2017), 1448-1487.
\bibitem[CEF]{CEF} T. Church, J.\,S. Ellenberg and B. Farb, \textit{$\FI$-modules and stability for representations of symmetric groups}, Duke Math. J. 164, no. 9 (2015), 1833-1910.
\bibitem[D]{D} P. Diaconis, \textit{The cutoff phenomenon in finite Markov chains}, Proceedings of the National Academy of Sciences, 93(4) (1996), 1659-1664.
\bibitem[L]{L} H. Lacoin, \textit{A product chain without cutoff}, Electron. Commun. Probab.
Volume 20 (2015), paper no. 19, 9 pp.
\bibitem[LS]{LS} E. Lubetzky and A. Sly, \textit{Cutoff phenomena for random walks on random regular graphs}, Duke Mathematical Journal,  153 (2010), 475-510.
\bibitem[LPW]{LPW} D. Levin, Y. Peres, and E. Wilmer \textit{Markov chains and mixing times}. Vol. 107. American Mathematical Soc., 2017.
\bibitem[PeSo]{PeSo} Y. Peres, and P. Sousi, \textit{Mixing times are hitting times of large sets}, Journal of Theoretical Probability, 28(2), 488-519.
\bibitem[RW]{RW} E. Ramos and G. White, \textit{Families of nested graphs with compatible symmetric-group actions}, \arXiv{1711.07456}.
\bibitem[RW2]{RW2} E. Ramos and G. White, \textit{Families of Markov chains with compatible symmetric-group actions}, \arXiv{1810.08475}.
\bibitem[RSW]{RSW} E. Ramos, D Speyer, and G. White, \textit{FI-sets with relations}, \arXiv{1804.04238}.

\end{thebibliography}
\end{document}